\documentclass[11pt]{amsart}

\usepackage{hyperref}

\newtheorem{thm}{Theorem}[section]

\newtheorem{lem}[thm]{Lemma}

\theoremstyle{definition}

\theoremstyle{remark}
\newtheorem{rem}[thm]{Remark}

\newtheorem{exe*}[thm]{Exercise*}
\newtheorem{exe!}[thm]{Exercise(!)}
\numberwithin{equation}{section}
\newcommand{\set}[1]{\left\{#1\right\}}
\newcommand{\Real}{\mathbb R}

\newcommand{\Natural}{\mathbb N}

\newcommand{\such}{\ | \ }
\newcommand{\ud}{\mathrm d}
\newcommand{\dfn}{:=}

\newcommand{\F}{\mathcal{F}}
\newcommand{\G}{\mathcal{G}}

\newcommand{\absco}{{<\kern-0.53em<}}

\newcommand{\e}{\mathrm{e}}

\newcommand{\pare}[1]{\left(#1\right)}
\newcommand{\bra}[1]{\left[#1\right]}

\newcommand{\indic}{\textbf{1}}

\newcommand{\dbracc}[1]{[\kern-0.15em[ #1 ]\kern-0.15em]}
\newcommand{\dbraoc}[1]{]\kern-0.15em] #1 ]\kern-0.15em]}
\newcommand{\dbraco}[1]{[\kern-0.15em[ #1 [\kern-0.15em[}
\newcommand{\dbraoo}[1]{]\kern-0.15em] #1 [\kern-0.15em[}

\newcommand{\filt}[1]{\mathcal{#1}_\cdot}

\renewcommand{\P}{\textsf{\upshape P}}
\newcommand{\Qu}{\textsf{\upshape Q}}

\newcommand{\E}{\textsf{\upshape E}}

\newcommand{\cF}{\mathcal{F}}
\newcommand{\cG}{\mathcal{G}}

\newcommand{\loct}{\Lambda}

\allowdisplaybreaks

\date{\today}

\title{Projections of Scaled Bessel Processes}
\author{Constantinos Kardaras \and Johannes Ruf}%
\email{k.kardaras@lse.ac.uk}%
\email{j.ruf@lse.ac.uk}%


\date{\today}%
\thanks{The authors would like to thank anonymous referees for their valuable comments. A shortened version of this article been submitted for journal publication.}
\begin{document}

\begin{abstract}
Let $X$ and $Y$ denote two independent squared Bessel processes of dimension $m$ and $n-m$, respectively, with $n\geq 2$ and $m \in [0, n)$, making $X+Y$ a squared Bessel process of dimension $n$. For appropriately chosen function $s$, the process $s (X+Y)$ is a  local martingale. We study the representation and the dynamics of $s(X+Y)$, projected on the filtration generated by $X$.  This projection is a strict supermartingale if, and only if, $m<2$. The finite-variation term in its Doob-Meyer decomposition only charges the support of the Markov local time of $X$ at zero.
\end{abstract}

\maketitle

\section*{Introduction} %\label{S:intro}

Optional projections of martingales are martingales; however, optional projections of local martingales are not necessarily local martingales. If the local martingale is nonnegative, \textsc{Fatou}'s lemma only yields that these optional projections are supermartingales.

Due to their analytic tractability, scaled \textsc{Bessel} processes of dimension two or higher are ideal to study this phenomenon. A first important step has been taken by \cite{Foellmer_Protter_2010} and \cite{Larsson_2013}, who consider the three-dimensional \textsc{Bessel} process, namely the modulus of a three-dimensional Brownian motion started away from zero, in the filtration generated by its components. The  reciprocal of the three-dimensional \textsc{Bessel} process is a local martingale; in \cite{Foellmer_Protter_2010} and \cite{Larsson_2013}, it is observed that its optional projection becomes a strict supermartingale when projecting  on the first component of the three-dimensional Brownian motion. However, when projecting on the first two components, the optional projection preserves the local martingale property.
 
In this article, we investigate these surprising observations further by providing a systematic study of optional projections of  scaled \textsc{Bessel} processes of any dimension greater than or equal to two. We provide two proofs to demonstrate our main result: one mostly analytic, while the other more probabilistic.
 
We need to point out the deep work of \cite{Carmona:Petit:Yor} on intertwining two related \textsc{Markov} processes. In particular, two squared \textsc{Bessel} processes of different dimensions are intertwined by an appropriate use of the expectation operator.  As already mentioned above, in this article, we complement their insights by focusing on the distinction between strict and non-strict supermartingales.

\section{Main Result} \label{S:main}
Consider a probability space $(\Omega, \cG, \P)$, equipped with two independent Brownian motions $B^X$ and $B^Y$.  Fix $n \geq 2$ and  $m \in [0, n)$ and consider the two stochastic differential equations
\begin{align*}
	X_t = 1 + m t + 2 \int_0^t \sqrt{X_u} \ud B^X_u&, \qquad t \geq 0;\\
	Y_t = (n-m) t + 2 \int_0^t \sqrt{Y_u} \ud B^Y_u&, \qquad t \geq 0.
\end{align*}
These stochastic differential equations have unique strong solutions, called squared \textsc{Bessel} process of dimension $m$ and $n-m$, respectively; see \cite[Section~XI.1]{RY}. 
 \textsc{L\'evy}'s characterisation of Brownian motion yields that $X+Y$ is also a squared \textsc{Bessel} process, now of dimension $n$.  \textsc{Feller}'s test for explosions yields that $X+Y$ is strictly positive since $n \geq 2$. We shall use $\filt{G}$ throughout to denote the natural filtration generated by the pair $(X, Y)$.

Next, consider the  function
\begin{align*}
s: (0, \infty) \ni w \mapsto 
\begin{cases}
	w^{1-n/2}, \qquad &\text{if $n > 2$};\\
	- \log(w), \qquad &\text{if $n = 2$}.
\end{cases}
\end{align*}
\textsc{It\^o}'s formula yields that $s (X + Y)$ is a  local martingale. Let $\filt{F}$ now denote the smallest right-continuous filtration that makes $X$ adapted. For future reference, note that the process $\int_0^\cdot \sqrt{X_u} \ud B^X_u$ is adapted to the filtration $\filt{F}$. We are interested in the $\filt{F}$--optional projection $Z$ of $s(X+Y)$, which is the unique  $\filt{F}$--optional process $Z$ such that 
\begin{align*}
Z_\tau = \E\left[s(X_\tau + Y_\tau) | \mathcal F_\tau \right]
\end{align*}
holds for all bounded $\filt{F}$ stopping times $\tau$.

\begin{rem} \label{F:1}
In order to ensure that $Z$ above exists, it suffices that $\E\left[|s(X_\tau + Y_\tau) | \right] < \infty$ holds for a fixed bounded $\filt{F}$ stopping time $\tau$. When $n > 2$, $\E\left[|s(X_\tau + Y_\tau) | \right] < \infty$ holds from the optional sampling theorem because $s(X+Y)$ is a nonnegative local martingale, thus a supermartingale, under $\filt{G}$. For $n=2$, we claim that $\E[|\log(J_\tau)|] < \infty$ for all bounded stopping times $\tau$ when $J$ is two-dimensional squared \textsc{Bessel} process with $J_0 = 1$. Indeed, first note that $\E[J_\tau] \leq 1 + 2 \E[\tau] < \infty$ holds from the dynamics of $J$, localisation, \textsc{Fatou}'s lemma and monotone convergence. Therefore, $\E[\log_+(J_\tau)] \leq \E[J_\tau] < \infty$ holds. Furthermore, since $\log J$ is a local martingale and $J_0 = 1$, we have $\E[\log_-(J_{\tau \wedge \tau_m})] = \E[\log_+(J_{\tau \wedge \tau_m})] \leq 1 + 2 \E[\tau \wedge \tau_m]$ along a localising sequence $(\tau_m)_{m \in \mathbb{N}}$, giving $\E[\log_-(J_{\tau})] \leq 1 + 2 \E[\tau] < \infty$ by \textsc{Fatou}'s lemma and monotone convergence.
\end{rem}

In order to set the stage for the statement of our main result, recall the Gamma function
\begin{align*}
(0, \infty)	\ni k \mapsto \Gamma(k) \dfn \int_0^\infty w^{k-1} \e^{-w} \ud w.
\end{align*}
Furthermore, define the stopping time
\begin{align} \label{eq:180225.7}
\rho \dfn \inf \left\{t \geq 0 \such X_t = 0\right\},
\end{align}
which is $\P$--almost surely finite when $0 \leq m < 2$.

In the case $0 < m < 2$, note that $X$ allows for \textsc{Markov} local time process $\loct$ at zero, defined via
\begin{equation} \label{eq:loct}
\loct_t \dfn \lim_{\varepsilon \downarrow 0} m \varepsilon^{-m/2} \int_0^t 
\indic_{\{X_u < \varepsilon  \}} \ud u, \qquad t \geq 0. 
\end{equation}
References for existence and properties of $\loct$ are provided in  Section~\ref{A:1} below; in particular, it will also be shown there that $\Lambda$ coincides with the semimartingale local time at zero of the scaled process $X^{1 - m/2} / (2 - m)$.

With the above notation, we now present the main result of this note.
\begin{thm} \label{T:1}
The $\filt{F}$--optional projection $Z$ of $s(X + Y )$ exists and satisfies $Z_t = f(t, X_t)$ for all $t > 0$, where
\begin{align}  \label{eq:180225.1}
f(t,x) \dfn \frac{1}{\Gamma\left( (n-m) / 2\right)} \times
\begin{cases}
\int_0^\infty (x + 2tw)^{1 - n/2} w^{(n-m)/2 - 1} \e^{-w} \ud w, \quad &\text{if $n > 2$}; \\
\int_0^\infty - \log(x + 2tw)w^{-m/2} \e^{-w} \ud w, \quad &\text{if $n = 2$}
\end{cases}
\end{align}
for all $t > 0$ and $x \geq 0$. Furthermore, the following statements hold:
	\begin{itemize}
		\item If $m \geq 2$ (thus, $n>2$), then
	\begin{align}  \label{eq:180228.5}
		Z = 1 + 2 \int_0^\cdot f'_x(u, X_u) \sqrt{X_u} \ud B^X_u;
	\end{align}
	hence $Z$ is a strict local martingale.
	\item If $m \in (0,2)$, then $Z$ is a strict supermartingale, that is, not a local martingale. With $\loct$ given by \eqref{eq:loct}, the \textsc{Doob-Meyer} decomposition of $Z$ is 	
	\begin{align} \label{eq:180228.4}  
Z = 
\begin{cases}
1 + 2 \int_0^\cdot f'_x(u, X_u) \sqrt{X_u} \ud B^X_u - \frac{\Gamma\left( m / 2\right)  }{\Gamma\left( n / 2 - 1 \right) } \int_0^\cdot \left( 1 / 2u\right)^{(n-m)/2} \ud \loct_u,\quad &\text{if $n > 2$}; \\
2 \int_0^\cdot f'_x (u, X_u) \sqrt{X_u} \ud B^X_u - \Gamma\left( m / 2 \right) \int_0^\cdot \left( 1 / 2u \right)^{1-m/2} \ud \loct_u,\quad &\text{if $n = 2$}.
\end{cases}
\end{align} 
	\item 
	If $m = 0$, then $Z$ is again a strict supermartingale of the form
\begin{align} \label{eq:180228.8}  
Z = 
\begin{cases}
1 + 2 \int_0^{\rho \wedge \cdot} f'_x(u, X_u) \sqrt{X_u} \ud B^X_u - \frac{2}{ \Gamma\left( n / 2 - 1 \right) } \int_0^\cdot \left( 1 / 2u\right)^{n/2} \indic_{\{\rho \leq u \}} \ud u,\quad &\text{if $n > 2$}; \\
2 \int_0^{\rho \wedge \cdot}f'_x  (u, X_u) \sqrt{X_u} \ud B^X_u -  \int_0^\cdot \left( 1 / u \right)  \indic_{\{\rho \leq u \}} \ud u,\quad &\text{if $n = 2$}.
\end{cases}
\end{align} 
\end{itemize}
\end{thm}

Section~\ref{S:2} contains a mostly analytic proof of Theorem~\ref{T:1}. Section \ref{S:3} contains an alternative proof, using more probabilistic arguments, for the case $n > 2$. This alternative route provides further intuition on the appearance of the local time in the \textsc{Doob-Meyer}  decomposition of $Z$. (Furthermore, this alternative route helped us to formulate the precise statements of Theorem~\ref{T:1}.) Lemma~\ref{L:180223} in Section~\ref{A:1} below summarises some results concerning the \textsc{Markov} local time process $\loct$, appearing in  \eqref{eq:180228.4}.

\begin{rem}
Here is a quick argument why $Z$ is a strict supermartingale if $m \in [0,2)$ and $n>2$.
  In general, the strict supermartingale property of $Z$ will follow from the non-constant finite-variation terms in \eqref{eq:180228.4} and \eqref{eq:180228.8} in the \textsc{Doob-Meyer} decomposition of $Z$.  All these assertions shall be  argued in the proof of Theorem~\ref{T:1}.
Assume now that $0 \leq m < 2 < n$, and suppose (as we shall see, by way of contradiction) that $Z$ is a local martingale. Since $X$ and $Y$ are independent and since the function $s$ is decreasing, we have
\[
Z_t \leq \E [s(Y_t)] = f(t,0) < \infty, \qquad t > 0.
\]
Since $Z$ is additionally strictly positive (recall that $n > 2$ is assumed), hence bounded, it would then follow that $(Z_t)_{t > 0}$ is an actual martingale. This would imply by \textsc{Fatou}'s lemma (note that $t = 0$ was not covered) that $Z$ is an actual martingale. But this is impossible, since it would have constant expectation, meaning that $s(X+Y)$ also has constant expectation, contradicting the fact that it is a strict local martingale; see \eqref{eq:180311.3} below. Therefore, we obtain that $Z$ fails to be a local martingale whenever $0 \leq m < 2 < n$.
\end{rem}

\begin{rem}
The special cases $n=3$ and $m \in \{1,2\}$ in Theorem~\ref{T:1} are studied in \cite{Foellmer_Protter_2010} and \cite{Larsson_2013}. When $n = 3$ and $m = 1$, using \eqref{eq:180225.1} we obtain
\begin{align*} 
Z_t &= \int_0^\infty \frac{1}{  2t \sqrt{X_t + y}} \exp\left(-\frac{y}{2t}\right) \ud y \\
&= \exp\left(\frac{X_t}{2t}\right)  \int_{X_t}^\infty   \frac{1}{  2t \sqrt{ w}} \exp\left(-\frac{w}{2t}\right) \ud w\\
&= \frac{1}{\sqrt{t} }\exp\left(\frac{X_t}{2t}\right) \int_{\sqrt{X_t/t}}^\infty  \exp\left(-\frac{y^2}{2}\right) \ud y, \\
&=  \sqrt{\frac{2 \pi}{ t}} \exp\left(\frac{X_t}{2t}\right) \left(1 - \Phi\left(\sqrt{\frac{X_t}{t}}\right)\right), \qquad t \geq 0,
\end{align*}
where $\Phi$ denotes the cumulative normal distribution. Recall the discussion after \eqref{eq:loct}, note that $\loct$ in \eqref{eq:180228.4}  is the semimartingale local time of $\sqrt{X}$. In contrast, \cite{Foellmer_Protter_2010} uses Brownian local time. These local times differ by a factor of $2$; see  \cite[Exercise~VI.1.17]{RY}. This explains the slight difference in the presentation of the finite-variation part in   \eqref{eq:180228.4}  from its representation in \cite{Foellmer_Protter_2010}. 

When $n = 3$ and $m = 2$, we obtain
\begin{align*}
Z_t &= \frac{1}{  \sqrt{2 \pi t}} \int_0^\infty \frac{1}{  \sqrt{y (X_t + y)}} \exp\left(-\frac{y}{2t}\right) \ud y \\
&= \frac{1}{\sqrt{2 \pi t}} \exp\left(\frac{X_t}{4t}\right) \int_1^\infty \frac{1}{\sqrt{w^2-1}} \exp\left(-\frac{w X_t}{4t}\right) \ud y \\
&=  \frac{1}{\sqrt{2 \pi t}} \exp\left(\frac{X_t}{4t}\right) K_0\left(\frac{X_t}{4t}\right), \qquad t \geq 0,
\end{align*}
where 
\[
(0, \infty) \ni k \mapsto K_0(k) \dfn \int_1^\infty \frac{1}{\sqrt{y^2-1}} \e^{-y k} \ud y
\]
denotes the modified  \textsc{Bessel} function of the second kind of order zero.
\end{rem}

\begin{rem}
As pointed out by \cite{Shiga:Watanabe}, if $X$ and $Y$ are appropriately chosen squared radial \textsc{Ornstein-Uhlenbeck} processes (of which squared \textsc{Bessel} processes are special cases) then so is $X+Y$. While it should be possible to extend the  arguments below to the case that  $X$, $Y$, and $X+Y$ are squared radial \textsc{Ornstein-Uhlenbeck} processes (such that $X+Y$ converted to natural scale is a local martingale), the notation would get unnecessarily complicated. We choose to sacrifice this bit of generality for more transparent formulas.
\end{rem}

\begin{rem} 
The function $f$ of \eqref{eq:180225.1} satisfies the partial differential equation
\begin{equation} \label{eq:180228.1}
f'_t(t,x) + m f'_x(t,x) + 2 x  f''_{x,x} (t,x) = 0, \qquad (t,x) \in (0,\infty)^2.
\end{equation}
This partial differential equation is derived from the assertion of Theorem \ref{T:1} via an application of \textsc{It\^o}'s formula to the local martingale $f(\cdot, X_{\rho \wedge \cdot})$---see Step 2 of the theorem's proof. The required derivatives of $f$ in \eqref{eq:180228.1} exist due dominated convergence.
\end{rem}

\section{Squared Bessel Processes and Their Markov Local Time} \label{A:1} 

We keep all notation from Section \ref{S:main}, and discuss here some useful properties of squared \textsc{Bessel} processes and their \textsc{Markov} local time.

\subsection{Facts concerning squared \textsc{Bessel} processes} \label{SS:2.1}

According to  \cite[Corollary~XI.1.4]{RY}, the process $Y$ has a density (with respect to \textsc{Lebesgue} measure), given by
\begin{equation} \label{eq:180225.4}
\P [Y_t \in \ud y] = \frac{1 }{\Gamma\left(\pare{n-m}/{2}\right) (2t)^{(n-m)/2}  }  y^{(n-m)/2 - 1} \exp\left(-\frac{y}{2t}\right) \ud y, \qquad t > 0, \quad y \geq 0.
\end{equation}
By \textsc{Feller}'s test of explosions, for $m \geq 2$, the process $X$ is strictly positive. For $m \in (0,2)$, $X$  visits level zero, but is instantaneously reflected there, i.e.,
\begin{equation} \label{180311.4}
\int_0^\cdot \indic_{\{X_t = 0\}} \ud t = 0.
\end{equation}
For $m = 0$, the process $X$ is absorbed when it hits zero. For a proof of these facts, see, for example, \cite[ Proposition~XI.1.5]{RY}  (but note that semimartingale local time is used there, while we shall only consider \textsc{Markov} local time---see Remark~\ref{R:Appendix} later on). When $m\in(0,2)$, the process $X$ accrues local time; i.e., with positive probability its \textsc{Markov} local time at zero is strictly positive. This is a consequence of Lemma~\ref{L:180223}.

Direct computations with the density of $X+Y$, or \cite[Example~1]{Kotani_2006}, yield that
\begin{equation}\label{eq:180311.3}
s(X+Y) \text{ is a strict local martingale in the  filtration } \filt{G}.
\end{equation}
Similarly, the process $X^{1-m/2}_{\cdot \wedge \rho}$ is a martingale. This yields that the $\filt{F}$ stopping time
\begin{align} \label{eq:180311.2}
\rho_\kappa \dfn \inf \left \{ t \geq 0 \ \Big| \ X_t \leq \frac{1}{\kappa} \right \},
\end{align}
where $\kappa>1$, has unbounded support. An alternative justification is provided in \cite[Corollary~1.2]{Bruggeman_Ruf}.

\subsection{\textsc{Markov} local time}

The next result discusses properties of local time of $X$ at zero.

\begin{lem}  \label{L:180223}
	Assume that $0 < m < 2$. Then the process $\loct$ defined via
	\[
	\loct_t \dfn \lim_{\varepsilon \downarrow 0} m \varepsilon^{-m/2} \int_0^t 
	\indic_{\{X_u < \varepsilon  \}} \ud u, \qquad t \geq 0.
	\]
	is a nondecreasing continuous additive functional, whose inverse, given by
 \begin{align} \label{eq:180225.6}
	A_s \dfn \inf \{t \geq 0: \loct_t> s\}, \qquad s \geq 0,
\end{align}	
	has (conditional) \textsc{Laplace} transform
\begin{align}  \label{eq:180223.1}
	\E \left[\e^{- z A_s}| \cF_{\rho} \right]  = \exp\left(-z\rho - \frac{\Gamma\left( m / 2 \right)} {\Gamma\left(1 -  m / 2 \right)} s \left(\frac{z}{2}\right)^{1 - m/2 }  \right), \qquad z \geq 0.
\end{align}
	Here, the stopping time $\rho$ is given in \eqref{eq:180225.7}.
Furthermore, we have
\begin{align}   \label{eq:180228.7}
\frac{1}{1 - m/2} X^{1-m/2} = \frac{1}{1 - m/2} +  2 \int_0^\cdot X^{(1-m)/2}_u \indic_{\{X_u > 0\}} \ud B^X_u + \loct.
\end{align}
\end{lem}

\begin{rem} \label{R:Appendix}
The process $\loct$ of Lemma~\ref{L:180223} is sometimes called ``\textsc{Markov}'' local time, in contrast to ``semimartingale'' local time, which only exists for semimartingales. For \textsc{Markov} semimartingales, these two local times may differ; however, as \eqref{eq:180228.7} shows, here the \textsc{Markov} local time $\loct$ of $X$ at zero is also the semimartingale local time of the process $X^{1-m/2} / (2-m)$ at zero. We refer to \cite{Gradinaru:Roynette:Vallois:Yor} and \cite{Donati-Martin:Roynette} for a deeper study of \textsc{Bessel} local time.
\end{rem}

\begin{proof}[Proof of Lemma~\ref{L:180223}]
We refer to \cite[Section~II.2 and Appendix~1.23]{Borodin_handbook}, where properties of $\loct$ are discussed, and further references are given.

The \textsc{Laplace} transform of $A$ is, to the best of our knowledge, first discussed in \cite{Molchanov:Ostrovskii}, but with $X$ replaced by $\sqrt{2 X}$ and some missing constants. To  argue \eqref{eq:180223.1},  note that 
\[
A_s = \rho + \inf \{t \geq 0: \loct_{t+\rho} - \loct_\rho > s\}, \qquad s \geq 0.
\]
Hence `Row~2' in
 \cite{Pitman:Yor:1999} yields
\begin{align*}
\E\left[\e^{- z A_s}| \cF_{\rho} \right]  &= \exp\left(-z \rho - s  \pare{\int_0^\infty  \e^{-z u}  \frac{1}{2^{m/2-1} u^{m/2} \Gamma( m / 2)}  \ud u}^{-1} \right)\\
		&= \exp\left(-z \rho - s  \frac{ 2^{m/2-1} \Gamma\left( m / 2 \right)   }{\int_0^\infty  \e^{-z u}   u^{1 - m/2  -1 }  \ud u} \right)\
		= \exp\left(-z \rho - s   \frac{ 2^{m/2-1} \Gamma\left( m / 2 \right) z^{1-m/2}  }{ \Gamma\left(1-  m / 2 \right) } \right),
\end{align*}
where we have used the transition density of $X$, provided in in \cite[Appendix~1.23 ]{Borodin_handbook} or in \cite[Corollary~XI.1.4 ]{RY}.

Continuing, the process $V \dfn X^{1-m/2} / (2-m)$  is a diffusion in natural scale; thus a semimartingale; see, for example, \cite[Lemma~5.22 ]{Assing:Schmidt}. 
The \textsc{Tanaka} formula then yields
\begin{align} 
	V &= \max(V, 0) = V_0 + \int_0^\cdot \indic_{\{V_u > 0\}} \ud V_u + \frac{1}{2}  l^0  =  V_0 + \int_0^\cdot \indic_{\{X_u > 0\}}
		 X^{(1-m)/2}_u \ud B^X_u +  \frac{1}{2}  l^0 \notag \\
		&=  V_0 + \int_0^\cdot \indic_{\{V_u > 0\}}  \left((2-m) V_u\right)^{(1-m)/(2-m)} \ud B^X_u + \frac{1}{2} l^0. \label{eq:180308.4}
\end{align}
Here, $(l^v)_{v \geq 0}$ denotes the semimartingale local time of $V$, continuous in time and right-continuous in the spatial variable $v \geq 0$, satisfying the \emph{occupations time formula}
	\begin{align*}
			\int_0^t g(V_u) \ud [V,V]_u  = \int_0^\infty g(v)  l^v_t \ud v, \qquad t \geq 0,
	\end{align*}
for all Borel--measurable functions $g: [0,\infty) \rightarrow [0, \infty)$; see \cite[Section~VI.1]{RY}.  Hence, we also have	
	\begin{align}  \label{eq:180308.2}
			\int_0^t g(V_u) \ud u =   \int_0^t g(V_u) \indic_{\{V_u > 0\}} \ud u  = \int_0^\infty g(v)  l^v_t \left((2-m)v\right)^{(2m-2)/(2-m)} \ud v
	\end{align}	
	for all Borel--measurable functions $g: [0,\infty) \rightarrow [0, \infty)$
	Here, the first equality follows from the fact that $X$, and hence $V$, are  \textsc{Lebesgue}--almost everywhere strictly positive, by \eqref{180311.4}. 
	Now, the continuity properties of $(l^v)_{v \geq 0}$ and \eqref{eq:180308.2} yield
	\begin{align*}
		l^0_t &= \lim_{\varepsilon \downarrow 0} \frac{1}{\int_0^\varepsilon  \left((2-m)v\right)^{(2m-2)/(2-m)} \ud v}
			\int_0^\varepsilon   l^v_t \left((2-m)v\right)^{(2m-2)/(2-m)} \ud v\\
		&= \lim_{\varepsilon \downarrow 0} m  \left((2-m)\varepsilon\right)^{m/(m-2)} 
		  \int_0^t \indic_{\{V_u < \varepsilon\}} \ud u \\
		&= \lim_{\varepsilon \downarrow 0}  m  \left((2-m)\varepsilon\right)^{m/(m-2)} \int_0^t 
			\indic_{\{X_u < ((2-m)\varepsilon)^{2/(2-m) }  \}} \ud u 
		= \lim_{\varepsilon \downarrow 0} m    \varepsilon^{-m/2} \int_0^t 
			\indic_{\{X_u < \varepsilon  \}} \ud u \\
		&=  \loct_t, \qquad t \geq 0,
	\end{align*}
where the last equality follows from the definition of $\loct$. Then, \eqref{eq:180228.7} follows from \eqref{eq:180308.4} and the above equality.
\end{proof}

\section{A Mostly Analytic Proof of Theorem~\ref{T:1}}  \label{S:2}

\subsection{Three technical lemmas}
Before we embark on proving Theorem \ref{T:1}, we shall provide some auxiliary analytic results.

\begin{lem}  \label{L:180228.1}
Assume that $m \in [0,2)$ and recall the function $f$ from \eqref{eq:180225.1}. Then, with
\begin{align} 
(0, \infty) \ni x \mapsto \psi (x) \dfn \frac{1}{\Gamma \left(\pare{n - m} / {2}\right)} \times 
\begin{cases}
\int_0^\infty \pare {1 - \pare{{w} / \pare{1 + w}}^{n/2 -1}} w^{- m/2} \e^{-x  w} \ud w, \quad &\text{if $n>2$}; \label{eq:180309.2} \\
\int_0^\infty \log\pare {1 + 1/w} w^{- m/2} \e^{-x  w} \ud w, \quad &\text{if $n=2$},
\end{cases}
\end{align}
it holds that $\psi \in C^\infty((0,\infty))$, that
\begin{align} \label{eq:180311.1}
f(t,x) = f(t, 0) - \frac{x^{1 - m / 2} }{(2t)^{(n- m )/2}} \psi \left(\frac{x}{2t}\right), \qquad t > 0, \quad x > 0,
\end{align}
and that
\begin{align} \label{eq:180311.1'}
\lim_{x \downarrow 0} \frac{x^{1 - m / 2} }{(2t)^{(n- m )/2}} \psi \left(\frac{x}{2t}\right) =0, \qquad t > 0.
\end{align}
\end{lem}

\begin{proof}
Let us  only consider the case $n>2$; the case $n=2$ follows in the same manner.
Since
\[
f(t, 2 t x) =  \frac{1}{\Gamma \left( (n - m)/2 \right)} (2t)^{1 - n/2 } \int_0^\infty \pare{\frac{w}{x + w}}^{n/2 -1} w^{- m/2 } \e^{-w} \ud w,
\]
for all $t > 0$ and $x > 0$,
we have 
\begin{align*}
(2t)^{n/2 - 1} \frac{f(t, 0) - f(t, 2 t x)}{x^{1 - m / 2}} &= \frac{1}{\Gamma \left( (n - m)/2 \right)}   \int_0^\infty \pare {1 - \pare{\frac{w}{x + w}}^{n/2 -1}} \pare{\frac{w}{x}}^{- m/2} e^{- w} \frac{\ud w}{x} \\
&=  \frac{1}{\Gamma \left( (n - m)/2 \right)}  \int_0^\infty \pare {1 - \pare{\frac{w}{1 + w}}^{n/2 -1}} w^{- m/2} e^{-x  w} \ud w
\\
&= \psi (x).
\end{align*}
Therefore, substituting $x$ for $2tx$, we obtain \eqref{eq:180311.1}. Finally, \eqref{eq:180311.1'} follows from the continuity of $f$ as seen easily in \eqref{eq:180225.1}.
\end{proof}

\begin{lem} \label{L:180309.1}
Assume that $m \in [0,2)$, and recall the function $\psi$ from \eqref{eq:180309.2}. Define the function
\begin{align*}
(0, \infty) \ni x \mapsto p (x) \dfn - x^{1 - m/2} \psi'(x).
\end{align*}
Then, $p$ is nonnegative and decreasing with $0 < p (0+) < \infty$. As a consequence, $\sup_{x > 0} p(x) < \infty$.
\end{lem}

\begin{proof}
We just consider the case $n>2$; the case $n=2$ follows in the same manner with the appropriate modifications. To simplify notation we shall consider the function $p_0 \dfn \Gamma((n-m)/2) p$.
Simple algebra and a change of variables gives
\begin{align*}
p_0 (x) &= \int_0^\infty \frac{1}{x} \pare {1 - \pare{\frac{w}{1 + w}}^{n/2 -1}} (x w)^{1 - m/2} e^{-x  w} \ud (x w) \\
&= \frac{1}{x} \int_0^\infty  \pare {1 - \pare{\frac{v}{x + v}}^{n/2 -1}} v^{1 - m/2} e^{-v} \ud v\\
&= \frac{1}{x} (L(0) - L(x)), \qquad x > 0,
\end{align*}
where 
\[
	L(x) =  \int_0^\infty  \pare{\frac{v}{x + v}}^{n/2 -1} v^{1 - m/2} e^{-v} \ud v, \qquad x > 0.
\]
Hence we get
\[
	p_0(x)  = -\int_0^1 L'(tx) \ud t =  \left(\frac{n}{2}-1\right) \int_0^\infty  \int_0^1\frac{1}{ \pare{tx + v}^{n/2}} v^{(n - m)/2} e^{-v} \ud t \ud v, \qquad x > 0.
\]
Thus, $p_0$ (and hence $p$) is nonnegative and decreasing with
\[
p_0(0+) =  \left(\frac{n}{2}-1\right) \int_0^\infty v^{ - m/2 } e^{-v} \ud v = 
 \left(\frac{n}{2}-1\right) \Gamma\left(1 - \frac{m }{2} \right).
\]
This concludes the proof.
\end{proof}

\begin{lem}  \label{L:180311.1}
Assume that $m \in (0,2)$. Then, we have
\begin{align*}
\psi(0) =
\begin{cases}
\frac{\Gamma\left(m / 2\right)} {\pare{1 - m / 2 } \Gamma\left( n / 2 - 1 \right)  }, \quad &\text{if $n>2$}; \\
\frac{\Gamma\left(m / 2\right)} {1 - m / 2  }, \quad &\text{if $n=2$}.
\end{cases}
\end{align*}
\end{lem}

\begin{proof}
Again, we only treat the case $n > 2$, as the case $n = 2$ can be argued in the same way. Straightforward computations yield
	\begin{align*}
		\int_0^\infty \pare {1 - \pare{\frac{w}{1 + w}}^{n/2 -1}} w^{- m/2} \ud w
			&=\left(\frac{n}{2} - 1\right) \int_0^\infty \int_w^\infty \pare{\frac{v^{n/2-2}}{(1 + v)^{n/2 }}}  \ud v w^{- m/2} \ud w\\
			&= \frac{ n / 2  - 1}{1 - m / 2} \int_0^\infty \frac{v^{(n-m)/2-1}}{(1 + v)^{n/2 }}  \ud v\\
			&= \frac{ n / 2  - 1}{1 - m / 2}  \int_0^1w^{(n-m)/2-1}  (1-w)^{m/2-1}  \ud w
			\\
			&= \frac{ n / 2  - 1}{1 - m / 2}  \frac{\Gamma\left(\pare{n-m} / 2 \right) \Gamma\left( m / 2\right)}{\Gamma\left( n / 2\right) } \\
			&= \frac{1}{1 - m / 2}  \frac{\Gamma\left(\pare{n-m} / {2}\right) \Gamma\left( m / 2 \right)}{\Gamma\left( n / 2 -1 \right) },
	\end{align*}
where we used the substitution $w = v/(1+v)$ in the third equality and 
 the identity
	\begin{align*}
		 \int_0^1 w^{a-1} (1-w)^{b-1} \ud u = \frac{\Gamma(a) \Gamma(b)}{\Gamma(a+b)} , \qquad a,b > 0
	\end{align*}
	in the fourth equality, 
which connects the Beta and Gamma functions. In the last equality of the long display, we have used the identity $\Gamma(k) = (k-1) \Gamma(k-1)$, which holds for all $k > 1$.
\end{proof}

\subsection{Proof of Theorem~\ref{T:1}} We proceed in several steps.

\bigskip

\noindent $\bullet$ \underline{Step 1}:
Using the density provided in \eqref{eq:180225.4}, we obtain
\begin{align*}
f(t,x) = \E[s(x + Y_t)],  \qquad t > 0, \quad x \geq 0,
\end{align*}
where the function $f$ is given in \eqref{eq:180225.1}. Note that the process $f(\cdot, X_\cdot)$ is $\filt{F}$--optional. Since we have already established the existence of the $\filt{F}$-optimal projection $Z$ of $s(X+Y)$ in Remark~\ref{F:1}, it immediately follows that $Z_t = f(t, X_t)$ holds for all $t \geq 0$.

\bigskip

\noindent $\bullet$ \underline{Step 2}:
Consider first the case $n>2$, fix some $\kappa > 1$, and recall the stopping times from \eqref{eq:180311.2}. Then $s(X^{\rho_{\kappa}} + Y^{\rho_{\kappa}})$ is bounded, hence a martingale under $\filt{G}$.  Its $\filt{F}$--optional projection, which is $Z^{\rho_\kappa}$, will also be a martingale. By \textsc{It\^o}'s formula and the fact that the derivatives of $f$ are continuous and the product \textsc{Lebesgue}$\otimes \P$ measure of $\set{(t, \omega): (t, X_t(\omega)) \in U}$ is strictly positive whenever $U$ is a nonempty open subset of $(0, \infty)^2$ due to the unbounded support of $\rho$ of \eqref{eq:180225.7}, the partial differential equation in \eqref{eq:180228.1} holds for all $(t,x) \in (0,\infty)^2$.

Let us now consider the case $n=2$ and fix again some $\kappa > 1$.  In this case,  \textsc{It\^o}'s formula yields
\begin{align*}
	s(X+ Y) = 2 \int_0^\cdot \frac{1}{\sqrt{X_u + Y_u}}  \ud W_u
\end{align*}
for some Brownian motion $W$. Hence, $s(X^{\rho_{\kappa}}+ Y^{\rho_{\kappa}})$ is a martingale under $\filt{G}$. Now we may conclude as in the case $n>2$ that the partial differential equation in \eqref{eq:180228.1} holds.

\bigskip

\noindent $\bullet$ \underline{Step 3}:
When $m \geq 2$, then $\lim_{\kappa \uparrow \infty} \rho_\kappa = \infty$ holds for the stopping times of \eqref{eq:180311.2}, thanks to the facts in \S \ref{SS:2.1}. Hence, $Z$ is indeed a local martingale satisfying \eqref{eq:180228.5} by \textsc{It\^o}'s formula and \eqref{eq:180228.1}.  It is, moreover, a strict local martingale since $s(X+Y)$ is not a martingale under $\filt{G}$, as noted in \eqref{eq:180311.3}.

\bigskip

\noindent $\bullet$ \underline{Step 4}: We now focus on the case $0 < m < 2$ and argue the finite-variation term appearing in the \textsc{Doob-Meyer} decomposition of $Z$ in \eqref{eq:180228.4}.
To make headway, Lemma~\ref{L:180228.1} yields
\[
Z_t = f(t, X_t) = 
 f(t, 0) - \frac{X_t^{1 - m / 2} }{(2t)^{(n- m )/2}} \psi \left(\frac{X_t}{2t}\right), \qquad t > 0.
\]
where the function $\psi$ is given in \eqref{eq:180309.2}. Unfortunately, since $\psi'(0) = -\infty$ by Lemma~\ref{L:180309.1}, we cannot use the product rule directly. Instead, we shall approximate the function $\psi$. For  $\varepsilon > 0$, define the function $\psi_\varepsilon: [0,\infty) \rightarrow \mathbb{R}$ by
$\psi_\varepsilon (x) = \psi(x)$ for all $x > \varepsilon$ and by
\[
\psi_\varepsilon(x) = \psi(\varepsilon) + \psi'(\varepsilon) (x - \varepsilon), \qquad x \in [0, \epsilon].
\]
Since $\psi$ is nonnegative, decreasing and convex, the same properties transfer to $\psi_\varepsilon$; furthermore, $\psi_\varepsilon \leq \psi$.

Next, fix some $t_0 > 0$. We shall first derive the dynamics of $Z$ for $t \geq t_0$ via approximation, and then send $t_0$ to zero. Given that $\psi_\varepsilon$ is convex and continuously differentiable on $[0, \infty)$, twice continuously differentiable except at $\varepsilon > 0$, and $\E \bra{\int_0^\infty \indic_{\{X_t = 2 \varepsilon t\}} \ud t} = 0$ holds,  it follows that $(\psi_\varepsilon( X_t/ (2t)))_{t \geq t_0}$ is a semimartingale satisfying
\begin{equation} \label{eq:180309.3}
\psi_\varepsilon \left(\frac{X_t }{2t}\right) =   \psi_\varepsilon \left(\frac{X_{t_0} }{2t_0}\right)  + \int_{t_0}^t \indic_{\{X_u  > 2u \varepsilon\}} \ud \psi \left(\frac{X_u }{2u}\right) + \psi'(\epsilon) \int_{t_0}^t  \indic_{\{X_u \leq 2u \varepsilon\}} \ud \left(\frac{X_u }{2u} \right), \qquad t \geq t_0.
\end{equation}
Define now the process
\begin{align} \label{eq:180309.4}
Z_t^\varepsilon \dfn  
 f(t, 0) - \frac{X_t^{1 - m / 2} }{(2t)^{(n- m )/2}} \psi_\varepsilon \left(\frac{X_t}{2t}\right), \qquad t > 0.
\end{align}
An application of \eqref{eq:180311.1} and integration-by-parts, in conjunction with \eqref{eq:180309.3} and \textsc{Tanaka}'s formula (see \eqref{eq:180228.7}), and recalling the partial differential equation \eqref{eq:180228.1} yield
\begin{align*}
Z^\varepsilon - Z^\varepsilon_{t_0} &= 2 \int_{t_0}^\cdot f'_x (u, X_u) \sqrt{X_u} \indic_{\{X_u > 2u \varepsilon\}}  \ud B^X_u \\
&+ \int_{t_0}^\cdot  f'_t (u, 0) \indic_{\{X_u  \leq 2 u \varepsilon\}}   \ud u
+ \frac{n-m}{2} \int_{t_0}^\cdot \frac{2 X_u^{1 - m / 2} }{(2u)^{(n- m )/2 + 1}} \psi_\varepsilon \left(\frac{X_u}{2u}\right) \indic_{\{X_u \leq 2u \varepsilon\}} \ud u
\\
&- (2-m) \int_{t_0}^\cdot  \frac{ \psi_\varepsilon \left( X_u / (2u) \right)}{(2u)^{(n- m )/2}}  X_u^{(1 - m) / 2}  \indic_{\{0 < X_u \leq 2 u \varepsilon\}} \ud B^X_u
 - \left(1-\frac{m}{2}\right) \psi_\varepsilon(0) \int_{t_0}^\cdot  \frac{\ud \loct_u}{(2u)^{(n- m )/2}}\\
&- \psi'(\epsilon) \int_{t_0}^\cdot \frac{X_u^{1 - m / 2} }{(2u)^{(n- m )/2}} \indic_{\{X_u \leq 2u \varepsilon\}} \ud \left(\frac{X_u }{2u} \right) \\
&- (2-m) \psi'(\epsilon) \int_{t_0}^\cdot \frac{2 X_u^{1 - m / 2}  }{(2u)^{(n- m )/2 + 1}} \indic_{\{X_u \leq 2u \varepsilon\}} \ud u.
\end{align*}
To derive this long display, two cases are considered. Whenever $X_u > 2 u \varepsilon$, then $Z^\varepsilon$ has the dynamics of a local martingale, provided in the first line of the long display. For the second case, namely when $X_u \leq 2 u \varepsilon$, we break up the second term on the right side of \eqref{eq:180309.4} in three components and apply the \textsc{It\^o} product rule. The second line of the long display provides the contribution of the first term in \eqref{eq:180309.4} and of the component involving the power of $t$.  The third line corresponds to the contribution of the power of $X$, after using the \textsc{Tanaka} formula in \eqref{eq:180228.7}.  The fourth line provides the contribution of $(\psi_\varepsilon( X_t / (2 t)))_{t \geq t_0}$, worked out in \eqref{eq:180309.3}. Finally, the last line yields the cross-product dynamics.

We now let $\varepsilon$ go to zero. Then $Z_t^\varepsilon$ tends to $Z_t$, for each $t > 0$.  Let us next consider the right side of the long display. Using \eqref{180311.4} and the bound $\psi_\varepsilon \leq \psi$, the dominated convergence theorem yields that the terms in the second line converge to zero. By a similar argument and \textsc{It\^o}'s isometry, so does the first term of the third line. For the fourth line, we bound the integrand
\[
\Bigg| \psi'(\epsilon) \frac{X_u^{1 - m / 2} }{(2u)^{(n- m )/2}}  \indic_{\{X_u \leq 2u \varepsilon\}} \Bigg| \leq 
\frac{ - \varepsilon^{1-m/2}  \psi'(\varepsilon)}{(2t_0)^{(n - m) /2- 1}}  \indic_{\{X_u \leq 2u \varepsilon\}} \leq \frac{ p(0+)}{(2t_0)^{(n - m) /2- 1}}  \indic_{\{X_u \leq 2u \varepsilon\}} 
\]
in the notation of Lemma~\ref{L:180309.1}. Hence, the term in the fourth line also converges to zero as $\varepsilon$ tends to zero. By exactly the same arguments, so does the term in the last line of the long display.

We are left with two terms. Consider the integral in the first line. Lemma~\ref{L:180228.1} yields that
\[
f'_x (t, x) = - \left(1 - \frac{m}{2}\right) \frac{x^{- m / 2} }{(2t)^{(n- m )/2}} \psi \left(\frac{x }{2t}\right) - \frac{x^{1 - m / 2} }{(2t)^{(n- m )/2}} \psi' \left(\frac{x }{2t}\right) , \qquad t > 0, \quad x > 0,
\]
so that $x  (f'_x(t, x))^2$ behaves like $k x^{1 - m} / t^{n-m}$ when $x \sim 0$, where $k > 0$ is an appropriate constant. However, $\int_0^\cdot X_u^{1-m} \ud u$ is a finite process, because it is (proportional to) the quadratic variation of the local martingale part in the dynamics of $X^{1 - m/2}$. Therefore, it follows that the integral in the first line converges to
\[
2 \int_{t_0}^\cdot f'_x(u, X_u) \sqrt{X_u} \indic_{\{X_u > 0\}}  \ud B^X_u .
\]
The only remaining term, namely the second term in the third line, converges to
\[
- \left(1-\frac{m}{2}\right) \psi (0) \int_{t_0}^\cdot  \frac{1}{(2u)^{(n- m )/2}} \ud \loct_u.
\]

To summarize, we have
\[
Z - Z_{t_0} = 2 \int_{t_0}^\cdot f'_x (u, X_u) \sqrt{X_u} \indic_{\{X_u > 0\}}  \ud B^X_u  - \left(1-\frac{m}{2}\right) \psi (0) \int_{t_0}^\cdot  \frac{1}{(2u)^{(n- m )/2}} \ud \loct_u.
\]
We can now sent $t_0$ to zero, noting that none of the integrals will explode because $X$ is away from zero on the stochastic interval $[0, \rho_2]$, where $\rho_2$ is given as in \eqref{eq:180311.2} with $\kappa = 2$.
 It then follows that
\[
Z = 1 + 2 \int_{0}^\cdot f'_x (u, X_u) \sqrt{X_u} \indic_{\{X_u > 0\}}  \ud B^X_u  - \left(1-\frac{m}{2}\right) \psi (0)\int_{0}^\cdot  \frac{1}{(2u)^{(n- m )/2}} \ud \loct_u.
\]
In conjunction with Lemma~\ref{L:180311.1}, this then yields \eqref{eq:180228.4}.

\bigskip

\noindent $\bullet$ \underline{Step 5}: 
Finally, for the case $m=0$ basic computations with \eqref{eq:180225.1} yield \eqref{eq:180228.8}. 
Indeed, if $n > 2$ we have
\[
f(t,x) = \frac{1}{\Gamma(n/2)} \int_0^\infty (x + 2tw)^{1-n/2} w^{n/2-1} \e^{-w} \ud w , \qquad t > 0, \quad x \geq 0.
\]		
This gives directly
\[
f(t,0) = \frac{(2t)^{1-n/2}}{\Gamma(n/2)} \int_0^\infty w^{1-n/2} w^{n/2-1} \e^{-w} \ud w  = \frac{(2t)^{1-n/2}}{\Gamma(n/2)}, \qquad t > 0.  
\]		
One then concludes by observing that $X$ gets absorbed when hitting zero; hence
\[
f(t, X_t) = f(t \wedge \rho, X_{t \wedge \rho}) + \indic_{\{\rho < t\}} (f(t,0) -f(\rho,0)). 
\]
The case $n=2$ is argued again in exactly the same way. \qed

\section{A Mostly Probabilistic Proof of Theorem~\ref{T:1}}  \label{S:3}
In this section we only consider the case $n > 2$, and provide an alternative proof of Theorem~\ref{T:1} in \S\ref{SS:3.2},
after some prerequisites for a certain dominating probability measure in \S\ref{SS:3.1}.

\subsection{A dominating probability measure in the canonical space}  \label{SS:3.1}
Note that the assertions of Theorem~\ref{T:1} only depend on the joint law on the path space of the process couple $(X,Y)$. Indeed, the local time $\Lambda$ only depends on $X$ and, moreover, we may write 
\[
2  \int_0^\cdot f'_x(u, X_u) \sqrt{X_u} \ud B^X_u =  \int_0^\cdot f'_x(u, X_u) \ud (X_u - mu)
\]
for the stochastic integral appearing in \eqref{eq:180228.5}--\eqref{eq:180228.8}. Hence, throughout Section~\ref{S:3}, we may and shall work on the space of functions from $[0, \infty)$ to $[0, \infty)^2$, where $(0, 0)$ will be a ``cemetery'' state and functions will be continuous in the open interval before hitting $(0, 0)$. More precisely, given a  right-continuous function $\omega: [0, \infty) \mapsto [0, \infty)^2$, define
\[
\zeta (\omega) \dfn \inf \set{t \geq 0 \such \omega(t) = (0,0)},
\]
and let $\Omega$ denote the set of all right-continuous functions $\omega: [0, \infty) \mapsto [0, \infty)^2$ such that $\omega(0) = (1,0)$, $\omega$ is continuous on $[0, \zeta(\omega))$, and $\omega (t) = (0, 0)$ holds for all $t \in [\zeta(\omega), \infty)$. Writing $\omega(t) = (\omega_x(t), \omega_y(t)) \in [0, \infty)^2$ for all $t \geq 0$, we define a pair of processes $(X, Y)$ via $X(\omega, t) = \omega_x(t)$ and $Y(\omega, t) = \omega_y(t)$ for all $(\omega, t) \in \Omega \times [0, \infty)$. We let $\G_\cdot$ be the right-continuous augmentation of the natural  filtration generated by $(X, Y)$, and note that $\zeta$ is a $\G_\cdot$ stopping time. Furthermore,  $\F_\cdot$ will be the right-continuous augmentation of the smallest filtration which makes $X$ adapted and $\zeta$ a stopping time. With $\G_\infty \dfn \bigvee_{t \geq 0} \G_t$ we let $\P$ denote the probability on $(\Omega, \G_\infty)$ under which $\P[\zeta < \infty] = 0$, and the coordinate process $(X, Y)$ consists of two independent squared \textsc{Bessel} processes, with dimensions $m$ and $n-m$ respectively, and recall that $X_0 = 1$ and $Y_0 = 0$ identically hold by the construction of $\Omega$.

We set $Z$ equal to the $(\filt{F}, \P)$--optional projection of $s (X+Y)$.  We recall that $s(X + Y)$ is a strictly positive $
\filt{G}$--local martingale, with localising sequence $(\tau_{1/n})_{n \in \Natural}$, where
\[
\tau_\kappa \dfn \inf \set{t \geq 0 \such X_t + Y_t \leq \kappa}, \quad \kappa \geq 0,
\]
the latter localising sequence having $\P$--almost sure limit $\tau_0$, which coincides $\P$--almost surely with $\zeta$. Hence we may use \textsc{F\"ollmer}'s construction and obtain a probability measure $\Qu$ on $(\Omega, \G_\infty)$ such that $\Qu[\tau_0 = \zeta] = 1$, and
\begin{equation} \label{eq:foellmer}
\E_\P \bra{V_\tau s(X_\tau + Y_\tau) \indic_{\set{\tau < \zeta}}} = \E_\Qu \bra{V_\tau \indic_{\set{\tau < \zeta}}},
\end{equation}
valid for any $\G_\cdot$ stopping time $\tau$ and nonnegative $\G_\cdot$--optional process $V$; see \cite{F1972} and \cite{Perkowski_Ruf_2014}.\footnote{To see that \textsc{F\"ollmer}'s construction can indeed be used as suggested, equip the space $E \dfn [0, \infty)^2 \setminus \{(0,0)\}$ with the metric 
	\[
	d(x,y) = \|x - y\| + \left|\frac{1}{\|x\|} - \frac{1}{\|y\|}\right|, \qquad (x,y) \in E \times E,
	\]
	where $\| \cdot \|$ denotes Euclidean norm in $\Real^2$.  Then $E$ equipped with the topology stemming from the previous metric is a Polish space, and the point $(0,0) \notin E$ can be identified with a cemetery state.
}
In particular, \eqref{eq:foellmer} above gives $\Qu [\tau_{1/\kappa} < \zeta] = 1$ for all $\kappa > 0$, which together with $\Qu[\tau_0 = \zeta] = 1$ implies that $\zeta$ is $\Qu$--almost surely equal to a $\filt{G}$--predictable stopping time.
As shown in the next result,  the $(\filt{F}, \P)$--local martingale property of $Z$ is related to whether $\zeta$ is $\Qu$--almost surely equal to an $\filt{F}$--predictable stopping time.

\begin{lem}\label{L:180225.2}
The following statements hold:
\begin{enumerate}	
	\item Write $\indic_{\dbraco{0, \zeta}} = L (1-K)$, where $K$ is $\filt{F}$--predictable, nondecreasing, with $K_0 = 0$, and $L$ is an $(\filt{F}, \Qu)$--local martingale. Then the process $Z / (1-K)$ is an $(\filt{F}, \P)$--local martingale.  
	\item If $\zeta$ is $\Qu$--almost surely equal to an $\filt{F}$--predictable  stopping time, then $Z$ is an $(\filt{F}, \P)$--local martingale.  
\end{enumerate}
\end{lem}

\begin{proof}
Write $\indic_{\dbraco{0, \zeta}} = L (1-K)$ as in the statement, and set $N = 1 / s(X+Y)$. \textsc{Bayes}' formula (see, e.g., Theorem~5.1 in \cite{Ruf_hedging}) yields
\begin{align*}
Z_t = \frac{ \E_\Qu [N_t (1 / N_t) \indic_{\set{\zeta > t}} \such \F_t] }{\E_\Qu[N_t \indic_{\set{\zeta > t}} | \cF_t]} = \frac{\indic_{\set{\zeta > t}}}{\E_\Qu[N_t | \cF_t]} = (1-K_t) \frac{L_t}{\E_\Qu[N_t |  \cF_t]}, \qquad t \geq 0,
\end{align*}
where the denominator denotes optional projection.  Since $L$ is an $(\filt{F}, \Qu)$--local martingale, the last fraction in the displayed formula above is an $(\filt{F}, \P)$--local martingale; hence so is $Z/(1-K)$. 
	
To see the validity of statement (2), note that if $\zeta$ is $\Qu$--almost surely equal to an $\filt{F}$--predictable stopping time, then $K = 0$ and $L = 1$ hold on $\dbraco{0, \zeta}$; since $\P [\zeta < \infty] = 0$, it follows that $Z$ is an $(\filt{F}, \P)$--local martingale.
\end{proof}

The previous lemma also yields that $1-K$ is the finite-variation component in the multiplicative  $(\filt{F}, \P)$--\textsc{Doob-Meyer}  decomposition of the $(\filt{F}, \P)$--supermartingale $Z$.  For future reference, and with $\mathcal L(1-K) \dfn - \int_0^\cdot 1/(1-K_u) \ud K_u$ denoting the stochastic logarithm of $1-K$, note that \textsc{It\^o}'s product rule yields that $\mathcal L(1-K)$ is the finite-variation component in the additive  $(\filt{F}, \Qu)$--\textsc{Doob-Meyer}  decomposition of the $(\filt{F}, \Qu)$--supermartingale $\indic_{\dbraco{0, \zeta}}$. Similarly, and using the fact that $\P[\zeta = \infty] = 1$, we have $\int_0^\cdot Z_{u-} \ud \mathcal L(1-K)_u$ 
is the finite-variation component in the additive $(\filt{F}, \P)$--\textsc{Doob-Meyer}  decomposition of $Z$.

\subsection{Putting everything together---a more probabilistic proof of Theorem~\ref{T:1}}\label{SS:3.2}
Showing the cases $m = 0$, $Z = f(\cdot, X)$, and \eqref{eq:180228.5} for $m \geq 2$ is done exactly as in Section~\ref{S:2}. However, here is an alternative argument for the local martingale property of $Z$ if $m \geq 2$.
Define $\rho$ as in \eqref{eq:180225.7}, 
so that, $\Qu$--almost surely, $\rho \leq \tau_0 = \zeta$. Next, note that
\begin{align*}
\Qu\left[\rho < \zeta\right] = \Qu\left[\rho < \tau_0\right] &= \Qu\left[X_{\rho} = 0, Y_{\rho} > 0, \, \rho < \infty\right] \\
&= \Qu\left[X_{\rho} = 0, \, N_{\rho} > 0, \, \rho < \infty\right] \leq   \E_{ \P} \left[Z_{\rho}  \indic_{\set{X_{\rho} = 0} \cap \{\rho < \infty   \}}   \right] = 0,
\end{align*}
the latter following from the fact that, under $ \P$, $ X$ is an $m$--dimensional squared \textsc{Bessel} process, which never hits zero. Thus, $\zeta$ is $ \Qu$--almost surely equal to an ${\filt{F}}$--predictable stopping time, yielding that $Z$ is an $(\filt{F}, \P)$--local martingale by Lemma~\ref{L:180225.2}(2).

Let us now consider the case $m \in (0,2)$. We start by computing the additive $(\filt{F}, \Qu)$--predictable compensator of  $ \indic_{\dbraco{\zeta , \infty}}$. To this end, note that the inverse local time $A$ of $X$ at zero, given in \eqref{eq:180225.6}, is right-continuous. Hence we may define the time-changed filtration $\widehat{\filt{F}} = (\widehat{\cF}_s)_{s \geq 0} = (\cF_{A_s})_{s \geq 0}$. With all relationships that follow valid under $\Qu$, $\lambda_0 \dfn \loct_{\tau_0} = \loct_{\zeta}$ is an $\widehat{\filt{F}}$ stopping time and  
\[
\zeta = \tau_0 =  A_{\lambda_0} = A_{\loct_{\tau_0}}
\]
holds. Define also the processes $\widehat X = X_{A_\cdot}$, $\widehat Y = Y_{A_\cdot}$ and $\widehat N = N_{A_\cdot} = 1/s(\widehat X + \widehat Y)$. Since $\widehat{X} = 0$ we have $\widehat N = \widehat Y^{n/2 - 1}$.  Noting that $\tau_0$ may only happen at times that $\loct$ charges, we conclude that if $\widehat F$ is the $( \widehat{\filt{F}}, \Qu)$--compensator of $\indic_{\dbraco{\lambda_0, \infty}}$, then $F = \widehat F_{\loct}$ is the  $({\filt{F}}, \Qu)$--compensator of $\indic_{\dbraco{\zeta, \infty}}$.

Fix now $u \geq 0$ and $h > 0$. In preparation for the calculations below, note that, since $X$ is independent of $Y$ under $\P$, the scaling property of the squared \textsc{Bessel} processes $Y$ starting from zero (e.g., this can be seen from the density in \eqref{eq:180225.4}) will give
\[
\E_{ \P} \left[\widehat{Z}_u \left| \widehat{\cF}_u\right.\right] = \E_{ \P} \left[Y^{1-n/2}_{A_u} \left| \widehat{\cF}_u\right.\right] = \E_{ \P}  \bra{Y^{1-n/2}_1} A_u^{1-n/2}.
\]
From this observation it follows that
\[
\E_{ \P} \left[\widehat{Z}_{u+h} \left| \widehat{\cF}_u\right.\right] = \E_{ \P} \left[ \E_{ \P} \left[\widehat{Z}_{u+h} \left| \widehat{\cF}_{u+h}\right.\right] \left| \widehat{\cF}_u\right.\right] = \E_{ \P}  \bra{Y^{1-n/2}_1} \E_{ \P} \left[A_{u+h}^{1-n/2} \left| \widehat{\cF}_u\right.\right],
\]
which in turn gives

\begin{align} \label{eq:190514}
\frac{\E_{ \P} \left[\widehat{Z}_{u+h} \left| \widehat{\cF}_u\right.\right]}
{\E_{ \P} \left[\widehat{Z}_{u} \left| \widehat{\cF}_u\right.\right]} = \E_{ \P} \left[ \frac{A_{u + h }^{1-n/2}}{A_{u}^{1-n/2}} \left| \widehat{\cF}_u\right.\right] = \left.\E_{ \P} \left[\left(\frac{x}{x+A_h - {\rho} }\right)^{n/2-1}\right] \right|_{x = A_u},
\end{align}
the last equality following from the fact that $A$ is an additive functional.
As a consequence of the previous calculations, on $\{\lambda_0 > u\} = \{\loct_{\tau_0} > u \}$ we have
\begin{align}
\Qu\left[ \lambda_0 > u + h \left| \widehat{\cF}_u\right.\right] &= \Qu\left[\widehat{Y}_v > 0 \text{ for all  } v \in (u, u + h]  \left| \widehat{\cF}_u\right.\right]  =  \Qu \left[\widehat{Y}_{u+h} > 0 \left| \widehat{\cF}_u\right.\right] \nonumber \\
&= \frac{\E_{ \P} \left[\widehat{Z}_{u+h} \left| \widehat{\cF}_u\right.\right]}
{\E_{ \P} \left[\widehat{Z}_{u} \left| \widehat{\cF}_u\right.\right]} = \left.\E_{ \P} \left[\left(\frac{x}{x+A_h - {\rho} }\right)^{n/2-1}\right] \right|_{x = A_u}.  \label{eq:180430.1}
\end{align}
Here the second equality follows from $\Qu[\tau_0 = \zeta] = 1$, which implies that zero is an absorbing state for $\widehat Y$ under $\Qu$. The second equality comes from \textsc{Bayes}' formula and the third equality is \eqref{eq:190514}.

In order to evaluate the last quantity, we shall use the following identity for the Gamma function:
\begin{align*}
y^{n/2-1} =  \frac{1}{\Gamma\left( n / 2 -1\right)} \int_0^\infty z^{n/2-2} \e^{-z/y} \ud z, \qquad y > 0.
\end{align*}
The \textsc{Laplace} transform in \eqref{eq:180223.1} then yields, for $x > 0$ and $h \geq 0$, that
\begin{align}
\E_\P\left[\left(\frac{x}{x+A_h - \rho}\right)^{n/2-1}\right] 
&=  \frac{1}{\Gamma\left(n/2-1\right) } \int_0^\infty  z^{n/2-2}   \e^{-z} \E_\P\left[\exp\left(-\frac{z}{x} (A_h - \rho)\right )\right]\ud z  \nonumber\\
&=   \frac{1}{\Gamma\left(n/2-1\right) }   \int_0^\infty       z^{n/2-2} \exp\left(-z -\frac{\Gamma\left( m / 2 \right)}{\Gamma\left(1 - m/2  \right )} \left(\frac{z}{2x}\right)^{1 - m/2 } h  \right) \ud z \nonumber\\
&= \int_0^\infty g(h, z; x) \ud z,
\label{eq:180430.2}
\end{align}
where we have used, for fixed $x > 0$, the function
\begin{align} \label{eq:180502.2}
[0, \infty)^2 \ni (h,z) \mapsto g (h, z; x) \dfn   \frac{1}{\Gamma\left( n / 2 -1\right)}       z^{n/2-2} \exp\left(-z -\frac{\Gamma\left(m / 2 \right)}{\Gamma\left(1 - m/2   \right )}   \left(\frac{z}{2x}\right)^{1 - m/2 }  h \right).
\end{align}
Note that, for all $h > 0$ and $z > 0$, 
\begin{align} \label{eq:180502.3}
0 \geq \frac{\partial} {\partial h} g (h, z; x) \geq  \left.\frac{\partial }{\partial h} g (h, z; x) \right|_{h  = 0+} = - \frac{\Gamma\left( m / 2 \right)}{\Gamma\left(n/2 -1\right) \Gamma\left(1 - m / 2    \right ) }       z^{n/2-2}  \left(\frac{z}{2 x}\right)^{1 - m/2 }\e^{-z}.
\end{align}
Since the right side of the last inequality is integrable, a combination of \eqref{eq:180430.1}, \eqref{eq:180430.2}, \textsc{L'H\^opital}'s rule, and the dominated convergence theorem give, for all $u \geq 0$, on the event  $\{\lambda_0 > u\}$,
\begin{align*}
\lim_{h \downarrow 0}  \frac{1}{h} \Qu\left[ \lambda_0 \leq u + h \left| \widehat{\cF}_u\right.\right] &= \lim_{h \downarrow 0} \frac{1}{h}    \left(1 - \int_0^\infty g(h, z; A_u) \ud z\right)
=  - \int_0^\infty  \left.\frac{\partial }{\partial h} g (h, z; A_u) \right|_{h  = 0+}  \ud z
\\
&=  \frac{\Gamma\left( m / 2 \right)}{\Gamma\left(n/2 -1\right) \Gamma\left(1 - m / 2    \right ) } \left(\frac{1}{2 A_u}\right)^{1 - m/2 }  \int_0^\infty     z^{(n-m)/2-1} \e^{-z} \ud z \notag\\
&=  \beta \left(\frac{1}{2 A_u}\right)^{1 - m/2 },
\end{align*}
where
\begin{align*}
\beta \dfn  \frac{\Gamma\left( m / 2 \right)  \Gamma\left(\pare{n-m}/2\right)}{\Gamma\left( n / 2 -1\right) \Gamma\left(1 -  m / 2    \right ) }. 
\end{align*}
It is now intuitively clear, and we in fact provide a precise argument at the end of the proof, that the additive $({\filt{F}}, \Qu)$--compensator  $\widehat F$ of $\indic_{\dbraco{\lambda_0, \infty}}$ has the form
\begin{align*}
\widehat F = \beta \int_0^{\cdot \wedge \lambda_0} \left(\frac{1}{2 A_s}\right)^{1-m/2} \ud s.
\end{align*}
Hence, in the notation of Lemma~\ref{L:180225.2}, and also in view of the discussion following it, the additive $({\filt{F}}, \Qu)$--compensator  $F = - \mathcal L(1-K)$ of $\indic_{\dbraco{\tau_0, \infty}} = \indic_{\dbraco{\zeta, \infty}}$ has the form
\begin{align*}
F = \beta \int_0^{\cdot \wedge \tau_0} \left(\frac{1}{2 u}\right)^{1-m/2} \ud \Lambda_u.
\end{align*}

Since
\[
\Gamma\left(\frac{n-m}{2}\right) f(u, 0) =  \int_0^\infty \frac{1}{ (2uw)^{n/2 - 1}  } w^{(n-m)/2 - 1} \e^{-w} \ud w = \frac{1}{ (2u)^{n/2 - 1}  } \Gamma \left(1 - \frac{m}{2}\right),
\]
the process
\[
- \int_0^\cdot Z_{u-} \ud F_u = -  \int_0^\cdot f(u, 0) \ud F_u  = - \frac{\Gamma\left( m / 2 \right)  }{\Gamma\left(n / 2 - 1\right) } \int_0^\cdot \frac{1}{(2u)^{(n-m)/2}} \ud \Lambda_u
\]
is the finite-variation component in the additive $(\filt{F}, \P)$--\textsc{Doob-Meyer}  decomposition of $Z$.  The fact that the martingale part will have the form as in \eqref{eq:180228.4} can be shown analytically, as in Section~\ref{S:2}.

It remains to argue that, for fixed $u \geq 0$ and $\Delta > 0$, on the event $\{ \lambda_0 > u\}$, we have
\begin{align} \label{eq:180502.1}
\Qu \left[\left. \lambda_0 \leq u +\Delta\right| \widehat{\F}_u \right] = 
\E_\Qu\left[ \left.  \beta \int_u^{u + \Delta} \left(\frac{1}{2 A_s}\right)^{1-m/2} \indic_{\set{\lambda_0 > s} } \ud s
\right| \widehat{\F}_u \right] .
\end{align}
To see this, note that, for each $N \in \mathbb{N}$, we have
\begin{align*}
\Qu \left[\left. \lambda_0 \leq u +\Delta\right| \widehat{\F}_u \right] 
&= \E_\Qu \left[ \left.  \sum_{n = 1}^{2^N}  \indic_{\set{ u + \pare{n-1} 2^{-N} \Delta < \lambda_0 \leq u + n 2^{-N} \Delta}} \right| \widehat{\F}_u \right] \\
&= \E_\Qu \left[ \left.  \sum_{n = 1}^{2^N}  \indic_{ \{ \lambda_0 > u + \pare{n-1} 2^{-N}  \Delta \}} \Qu\left[\left.  
\lambda_0 \leq u +\frac{n}{2^N} \Delta
\right| \widehat{\F}_{u + (n-1) 2^{-N} \Delta} \right] \right| \widehat{\F}_u  \right]\\
&= \E_\Qu \left[ \left. \int_u^{u + \Delta} \eta_u^N \ud u \right|\widehat{\F}_u \right],
\end{align*}	
for the piecewise constant process
\begin{align*}
\eta^N &\dfn   \sum_{n = 1}^{2^N}  \indic_{ \{ \lambda_0 > u + \pare{n-1} 2^{-N} \Delta \}} \left(2^N
\Qu\left[\left.  
\lambda_0 \leq u+\frac{n}{2^N} \Delta
\right|  \widehat{\F}_{u + \pare{n-1} 2^{-N} \Delta} \right]\right)
\indic_{\dbraco{u + \pare{n-1} 2^{-N} \Delta, u+ n 2 ^{-N}\Delta}} \\
&= \sum_{n = 1}^{2^N}  \indic_{ \{ \lambda_0 > u + \pare{n-1} 2^{-N} \Delta \}} 
\left(- \frac{\partial} {\partial h} \int_0^\infty g \left(h^N_n, z;  A_{u + (n-1) 2^{-N} \Delta} \right) \ud z \right) \indic_{\dbraco{u + \pare{n-1} 2^{-N} \Delta, u+ n 2 ^{-N}\Delta}},
\end{align*}
by \eqref{eq:180430.1}, \eqref{eq:180430.2}, and the mean value theorem, where $h_n^N$ is a $[0, 2^{-N} \Delta]$--valued  $\widehat{\F}_{u + (n-1) 2^{-N} \Delta}$--measurable random variable and the function $g$ is defined in \eqref{eq:180502.2}. Thanks to \eqref{eq:180502.3}, on the event $\{A_u > 1/\kappa\}$ for any $\kappa > 0$, the sequence  $(\eta^N)_{N \in \mathbb{N}}$ is uniformly bounded. It also satisfies
\begin{align*}
\lim_{N \uparrow \infty} \eta^N =  \indic_{\dbraco{u, \lambda_0 \wedge (u + \Delta)} }\, \beta \left(\frac{1}{2 A}\right)^{1 - m/2 },
\end{align*}
\textsc{Lebesgue}$\otimes \Qu$--almost everywhere. Hence, an application of the dominated convergence theorem yields the claim in \eqref{eq:180502.1}. This concludes the proof.
\qed

%----------------------------------------------------------------
\bibliographystyle{amsalpha}
\bibliography{aa_bib}
\end{document}